\documentclass[10pt,a4paper,norsk]{article}
          
\usepackage{mystyle}

\title{A toric proof of B\'ezout's theorem \\ for weighted projective spaces}    
\author{Bernt Ivar Utst$\o$l N$\o$dland}     
\date{\today}

\begin{document} 

\maketitle

\begin{abstract}
Using toric geometry we prove a B\'ezout type theorem for weighted projective spaces. 
\end{abstract}

\section{Introduction}

The classical B\'ezout's Theorem for projective space states that for divisors $D_1,...,D_n$ on $\p^n$ we have $D_1 \cdots D_n = \Pi_{i=1}^n \deg D_i$ \cite[Prop. 8.4]{Fulton2}. Using toric geometry we will generalize this formula to weighted projective spaces $\p(q_0,...,q_n)$:

\begin{theorem} [B\'ezout's Theorem] \label{Bezout}
Given n torus-invariant divisors $E_1,...,E_n$ on $\p(q_0,...,q_n)$, we have
\[ E_1 \cdots E_n = \frac {\Pi_{i=1}^n \deg E_i}{q_0 \cdots q_n} . \]
\end{theorem}

We recover B\'ezout's theorem for weighted homogeneous polynomials  \cite[Thm. 3.6]{Ertl} by choosing effective divisors intersecting in finitely many points with an additional hypothesis on the degrees of the polynomials defining the divisors. Our theorem also  generalizes B\'ezout's theorem for the weighted projective plane given in \cite[Prop. 5]{bezplane}.

This paper is  based on results from my master thesis \cite{biun} and is written as part of my PhD thesis at the University of Oslo, supervised by Ragni Piene and John Christian Ottem.

\section{Preliminaries on intersection theory and WPS}

For any $n$-dimensional variety $X$ let $Z_k(X)$ be the  free abelian group generated by the set of irreducible closed subvarieties of dimension $k$ on $X$. We define rational equivalence: Let $\alpha \in Z_k(X)$ be equivalent to zero if there exist finitely many $(k+1)$-dimensional subvarieties $V_i \subseteq X$ such that $\alpha$ is the divisor of a rational function on $V_i$ for all $i$. Then the $k$-th Chow group $A_k(X)$ is $Z_k(X)$ modulo rational equivalence.  

We use notation and definitions from \cite{Cox} for toric varieties. For a complex toric variety $X_\Sigma$ defined by a fan $\Sigma$, $A_k(X_\Sigma)$ is generated by the classes of the orbit closures $\overbar{O(\sigma)}$ of the cones $\sigma \in \Sigma(n-k)$ (\cite[Ch.5.1]{Fulton}). If $\Sigma$ is complete and simplicial, setting $A^k(X_\Sigma)=A_{n-k}(X_\Sigma)$, one can define a product
\[ A^k(X) \otimes \mathbb{Q} \times A^l(X) \otimes \mathbb{Q} \to A^{k+l}(X) \otimes \mathbb{Q} \]
which agrees with geometric intersection in nice cases (see \cite[Remark 10.9]{Danilov}). This makes the groups of cycles into a graded ring $A^\bullet(X_\Sigma)_{\mathbb{Q}}$. 

To compute intersections we will also consider the Chow ring of a toric variety, as defined in \cite[Ch. 12.5]{Cox}.

Given a fan $\Sigma$, let $\Sigma(1)=\{\rho_1,...,\rho_r \}$. Denote by $u_i$ the minimal generator of $\rho_i$. We will consider two ideals $\mathscr{I}$, $\mathscr{J}$ in the polynomial ring $\mathbb{Q}[x_1,...,x_r]$. Let
\[ \mathscr{I} = \langle x_{i_1} \cdots x_{i_s} | \text{ all } i_j \text{ distinct and } \rho_{i_1}+ \cdots + \rho_{i_s} \text{ is not a cone in } \Sigma \rangle , \]
\[ \mathscr{J} = \langle \sum_{i=1}^r \langle m,u_i \rangle x_i | \text{ where } m \text{ ranges over a basis of } M \rangle . \]
The ideal $\mathscr{I}$ is called the Stanley--Reisner ideal. The Chow ring $R_{\mathbb{Q}}(\Sigma)$ is defined as
\[ R_{\mathbb{Q}}(\Sigma)= \mathbb{Q} [x_1,...,x_r]/\mathscr{I}+\mathscr{J} . \]
We have that if $X_\Sigma$ is complete and simplicial, then by \cite[Thm 12.5.3]{Cox}
\[R_{\mathbb{Q}}(\Sigma) \cong A^\bullet(X_\Sigma)_{\mathbb{Q}} . \]
We have from \cite[Example 3.1.17]{Cox} the fan for a weighted projective space: Given natural numbers $q_0,...,q_n$ with $\gcd(q_0,...,q_n)=1$, consider the quotient lattice of $\Z^{n+1}$ by the subgroup generated by $(q_0,...,q_n)$. We write $N=\mathbb{Z}^{n+1} / \mathbb{Z} (q_0,...,q_n)$. Let $u_i$ for $i=0,...,n$ be the images in $N$ of the standard basis vectors of $\Z^{n+1}$. This means that in $N$ we have the relation
\[ q_0u_0 + ... + q_nu_n = 0 . \]
Let $\Sigma$ be the fan consisting of all cones generated by proper subsets of $\{ u_0,...,u_n \}$. Then $X_\Sigma = \mathbb{P}(q_0,...,q_n)$, which is complete and simplicial. Moreover, we have 
\[\mathscr{I} = \langle x_0 \cdots x_n \rangle . \]
Since we are  over $\mathbb{Q}$, a basis for $M \otimes \Q = \{ m \in \Z^{n+1} | \sum q_im_i = 0 \}$ will be $(q_i,0,...,0,-q_0,0...,0)$, for $i=1,...,n$. This gives the ideal
\[ \mathscr{J} = \langle q_ix_0 -q_0x_i | i=1,...,n \rangle . \]
Doing the computations, we can eliminate $x_1,...,x_n$ since $x_i = \frac{q_i}{q_0}x_0$, so the Chow ring will be
\[ R_{\mathbb{Q}}(\Sigma) \cong \mathbb{Q}[x_0]/x_0^{n+1} . \]
The group of torus-invariant divisors is the free abelian group generated by one prime divisor $D_i$ for each $1$-dimensional ray $\rho_i$ in $\Sigma$. The $1$-graded part of $R_{\mathbb{Q}}(\Sigma)$ corresponds to $\Cl(\p(q_0,...,q_n)) \otimes \Q$. It is well known that $\Cl(\p(q_0,...,q_n))$ is isomorphic to $\Z$ via the degree map, where $\deg D_i = q_i$ (see for instance \cite[Thm. 1.19]{RossiTerra}). Then we have relations $q_iD_0 = q_0D_i$, thus choosing the image of $D_0$ (by abuse of notation we denote this by $D_0$ as well) as a generator for  $\Cl(\p(q_0,...,q_n)) \otimes \Q \simeq \Q$, we get that $D_i$ is mapped to $\frac{q_i}{q_0} D_0$. Taking any torus-invariant divisor $D = \sum_{i=0}^n a_iD_i$, we have  $\deg D= \sum_{i=0}^n a_iq_i$. Then in the Chow ring, $D$ gets mapped to $\sum_{i=0}^n a_iD_i = \sum_{i=0}^n a_i \frac {q_i}{q_0}D_0 = \frac{D_0}{q_0} \sum_{i=0}^n a_iq_i = \frac{deg D}{q_0} D_0$.

Taking $n$ torus-invariant divisors $E_1,...,E_n$ it then follows
\[ E_1 \cdots E_n = \frac {\Pi_{j=1}^n \deg E_j}{q_0^n} D_0^n , \]
thus we have determined intersections of divisors modulo the generator $D_0^n$. We wish to push this forward to $\Spec \C$ to obtain an actual number, i.e., to calculate $\int_{X_\Sigma} D_0^n$.

If a  complete variety $X_\Sigma$ of dimension $n$ is embedded in $\p^s$ via a very ample divisor $D$, define  $D^n \defeq \deg (X_{\Sigma_P} \subset \p^s) = \int_{X_\Sigma} D^n$. Then  by \cite[Thm. 13.4.3]{Cox}, $D^n$ equals $\Vol(P_D)$, i.e., the volume of the polytope associated to the divisor, normalized with respect to the lattice $M$.

To apply this we need to to describe a polytope giving $\p(q_0,...,q_n)$. From \cite[Remark 1.24 and Cor 1.25]{RossiTerra} we have the following polytope:

Given $(q_0,...,q_n)$ and $M \cong \Z^{n+1}$, let $\delta = \lcm(q_0,...,q_n)$. Consider the $n+1$ points of $M_\R \cong \R^{n+1}$:
\[ v_i = (0,...,\frac{\delta}{q_i},...0)\]
Let $\Delta$ be the convex hull of $0$ and all $v_i$. Intersecting $\Delta$ with the hyperplane $H= \{ (x_0,...,x_n) | \sum_{i=0}^n x_iq_i = \delta \}$, we get an $n$-dimensional polytope $P$. Then $X_P \cong \p(q_0,...,q_n)$ and the divisor $D$ associated to the polytope will be  $\frac {\delta}{q_0}D_0$ which is very ample.

If we  can determine the  volume of $P$,  we can determine $D_0^n$, since
\[ \Vol(P) = D^n=\frac {\delta^n}{q_0^n}D_0^n , \]
implying that $D_0^n=\Vol(P) \frac{q_0^n}{\delta^n}$.

\section{Proof of Theorem \ref{Bezout}}

To determine the  volume of $P$, we will use the generalized cross product (see \cite{Cross}). For $n$ vectors $v_1,...,v_n \in \R^{n+1}$, let $A$ be the matrix with $i$-th row $v_i$. The cross product $v_1 \times  \cdots \times v_n \in \R^{n+1}$ has  $k$-th coordinate equal to $(-1)^k$ times the $n \times n$ minor of $A$ obtained by removing the $k$-th column. This cross product is orthogonal to all $v_i$ and satisfies
\[ |v_1 \times  \cdots \times v_n| = \Vol(v_1,...,v_n) , \]
where $\Vol(v_1,...,v_n)$ is the $n$-dimensional volume of the parallelotope spanned by $v_1,...,v_n$ (this product can be expressed by exterior algebra operations as the Hodge dual *$(v_1 \wedge \cdots \wedge v_n)$).

To determine the  volume, we first need to normalize with respect to the lattice, i.e., we need to determine the volume spanned by a basis. We will need the following to choose a basis for the lattice spanned by $P$: Given set of linearly independent vectors $b_1,...,b_n \in M$ let 
\[ T(b_1,...,b_n) = \{ \sum_{i=1}^n c_ib_i | 0 \leq c_i < 1 \} \subseteq M_{\mathbb{R}} = M \otimes \mathbb{R} .\]
The following is well-known, but we include a proof for lack of a proper reference:
\begin{lemma} \label{basis}
The vectors $b_1,...,b_n$ form a basis for the lattice $M$ if and only if $T(b_1,...,b_n) \cap M = \{ 0 \}$ .
\end{lemma}

\begin{proof}
Assume $(b_1,...,b_n)$ is a basis. Let $x \in T(b_1,...,b_n) \cap M$. Then $x= \sum_{i=1}^n c_ib_i = \sum_{i=1}^n n_ib_i$ for $0 \leq c_i < 1$, $n_i \in \mathbb{Z}$. Thus $0 = \sum_{i=1}^n (c_i-n_i)b_i$. Since the $b_i$'s are linearly independent, this implies that $c_i = n_i$, hence $c_i=0$.

Assume $T(b_1,...,b_n) \cap M = \{ 0 \}$. Pick a lattice point $x \in M$. Since $b_1,...,b_n$ is a basis for the vector space $M_{\mathbb{R}}$  we can find $d_i \in \mathbb{R}$ such that $x= \sum_{i=1}^n d_ib_i$. Let $d_i = n_i +c_i$ where $n_i \in \mathbb{Z}$ and $0 \leq c_i < 1$. Then $x - \sum_{i=1}^n n_ib_i \in T(b_1,...,b_n) \cap M = \{ 0 \}$, hence $c_i = 0$ for all $i$. Thus $b_1,...,b_n$ is a basis for $M$.
\end{proof}

  First we choose an edge of the polytope $P$, say the edge $v_0v_1$, which is generated by $(-\frac{\delta}{q_0},\frac{\delta}{q_1},0,...,0)$. For simpler notation set $q_{i_1,...,i_s}=\gcd(q_{i_1},...,q_{i_s})$. The primitive generator of the edge $v_0v_1$ will be $e_1=(-\frac{q_1}{q_{01}},\frac{q_0}{q_{01}},0,...,0)$. Now, choose any lattice point of $H = \{ (x_0,...,x_n) | \sum_{i=0}^n x_iq_i = \delta \}$ of the form
\[ (x_{20},x_{21},\frac{q_{01}}{q_{012}},0,...,0), \]
this exists since the numbers obtained as integral linear combination of $q_0,q_1$ are exactly all multiples of $q_{01}$, and $\delta-q_2 \frac{q_{01}}{q_{012}}$ is such a multiple (the subscripts are chosen for notational purposes which will become clear) . Set $e_2$ as the difference between this point and $v_0$, in other words
\[e_2=(x_{20}-\frac{\delta}{q_0},x_{21},\frac{q_{01}}{q_{012}},0,...,0 ) .\]
In general, for all $2 \leq s \leq n$ find a lattice point of the form
\[ (x_{i0},x_{i1},...,x_{i(s-1)},\frac{q_{0...s-1}}{q_{0...s}},0,...,0). \]
This is equivalent to saying
\[ x_{i0}q_0 + x_{i1}q_1 + \cdots + x_{i(s-1)}q_{s-1} + \frac{q_{0...s-1}}{q_{0...s}} q_s = \delta .\]
Set
\[ e_s = ( x_{i0}-\frac{\delta}{q_0},x_{i1},...,x_{i(s-1)},\frac{q_{0...s-1}}{q_{0...s}},0,...,0). \]
Then we have

\begin{proposition} \label{nbasis}
The $n$ vectors $\{ e_1,...,e_n \}$ constructed above, are a basis for the lattice spanned by $H$.
\end{proposition}

\begin{proof}
We will use Lemma \ref{basis} to show this. Assume we have a lattice point $l=\sum_{i=1}^n c_ie_i$, where $0 \leq c_i < 1$ for all $i$. Then it suffices to show that all $c_i=0$.
We will show this by descending induction on $c_n$. Let $l=(y_0,...,y_n)$. Then we have, by definition of $H$,
\begin{equation} \label{jau}
\sum_{i=0}^n q_iy_i=\delta .
\end{equation}
Consider the $(n+1)$-th coordinate. Since the basis is constructed in such a way that the only vector having nonzero $(n+1)$-th coordinate is $e_n$, we must have $y_n=c_n \frac{q_{0,...,n-1}}{q_{0,...,n}}$. When we defined weighted projective space, we assumed $q_{0,...,n}=1$. Thus we must have $y_n=c_n q_{0,...,n-1}$. Now consider \eqref{jau} modulo $(q_{0,...,n-1})$: The righthand side is $0$ and the first terms $q_0y_0+...+q_{n-1}y_{n-1}$ will be zero, since in general integral linear combinations of a set of integers are exactly the multiples of their greatest common divisor. Thus we must have
\[ q_ny_n \equiv q_nc_nq_{0,...,n-1} \equiv 0 \pmod{q_{0,...,n-1}}. \]
Now since $c_n < 1$, we have $c_nq_{0,...,n-1} < q_{0,...,n-1}$, and if $0<c_n$ there must be some prime power $p^r$ dividing $q_{0,...,n-1}$ which does not appear in $c_nq_{0,...,n-1}$. But then we must have that $p$ divides $q_n$, which implies $q_{0,...,n}>1$ which is a contradiction. Thus $c_n = 0$.

Assume in general we have proved that $c_n=c_{n-1}=...=c_{s+1}=0$. We will show that $c_s=0$. We will use the same method as above: Since $c_{s+1}=...=c_n=0$, we have a linear combination $l=\sum_{i=0}^s c_ie_i$. In the set $\{e_1,...,e_s \}$, the only vector with $(s+1)$-th coordinate nonzero will be $e_s$. Thus we must have $y_s=c_s \frac {q_{0,...,s-1}}{q_{0,...,s}}$. Considering \eqref{jau} modulo $q_{0,...,s-1}$ we get 
\[ q_sy_s \equiv q_sc_s \frac{q_{0,...,s-1}}{q_{0,...,s}} \equiv 0 \pmod{q_{0,...,s-1}}. \]
Now, since $l$ is a lattice point, $c_s \frac{q_{0,...,s-1}}{q_{0,...,s}}$ is an integer $k < \frac{q_{0,...,s-1}}{q_{0,...,s}} $. Rewriting the above we get
\begin{equation} \label{jups}
 \frac {q_s}{q_{0,...,s}} k q_{0,...,s} \equiv 0 \pmod{q_{0,...,s-1}} ,
\end{equation}
since $k q_{0,...,s}=c_s q_{0,...,s-1} < q_{0,...,s-1}$, we must have, if $0 < c_s$, that there is a prime power $p^r$ in the prime factorization of $q_{0,...,s-1}$, which appears to a smaller degree in the prime factorization of $c_s q_{0,...,s-1}$. By the previous equality, the highest power of $p$ which can appear in $q_{0,...,s}$ will also be smaller than $r$, say it is $(r-t)$. But to satisfy \eqref{jups} we must also have that $p$ divides $\frac{q_s}{q_{0,...,s}}$, which implies that $p^{r-t+1}$ divides $q_s$, but then $p^{r-t+1}$ will divide $q_{0,...,s}$ which is a contradiction. Thus we must have $c_s = 0$. 

The last case is an exception. If $s=0$ we have $l = c_0e_0$, but by construction of $e_0$ as a primitive vector we must have $c_0=0$. Hence we are done.
\end{proof}

Now we can use this to calculate the normalized volume.

\begin{proposition} \label{volume}
The volume of the parallelotope spanned by $e_1,...,e_n$ is $\sqrt{q_0^2+...+q_n^2}$.
\end{proposition}

\begin{proof}
The coordinates of $z=e_1 \times \cdots \times e_n$ will be (modulo a sign) the  $n \times n$ minors of the matrix $A$ with row $i$ equal to $e_i$.

\begin{align*}
A=
  \begin{bmatrix}
 -\frac{q_1}{q_{01}} & \frac{q_0}{q_{01}} & 0 & 0 & \cdots & 0\\ 
 x_{20}-\frac{\delta}{q_0} & x_{21} & \frac{q_{01}}{q_{012}} & 0 & \cdots &0  \\
x_{30}-\frac{\delta}{q_0} & x_{31} & x_{32} & \frac{q_{012}}{q_{0123}}  & \ddots &0  \\
\vdots  & \vdots & \vdots & \ddots & \ddots & 0 \\
x_{n0}-\frac{\delta}{q_0} & x_{n1} & x_{n2} & x_{n3} & \cdots & \frac{q_{0,...,n-1}}{q_{0,...,n}}
\end{bmatrix}
\end{align*}

Set $z=(z_0,...,z_n)$. We see immediately that  $z_0=q_0$ and $z_1=q_1$, since the corresponding minors are lower triangular and $q_{0,...,n}=1$. To calculate $z_s$ we get, by expanding along the columns from the right, that $z_s=(-1)^sq_{0,...,s}\det(B_s)$ where $B_s$ is the $s \times s$ submatrix from the upper left corner of $A$. Consider such a $B_s$:

\begin{align*}
B_s=
  \begin{bmatrix}
 -\frac{q_1}{q_{01}} & \frac{q_0}{q_{01}} & 0 & 0  & \cdots & 0\\ 
 x_{20}-\frac{\delta}{q_0} & x_{21} & \frac{q_{01}}{q_{012}} & 0 & \cdots &0  \\
x_{30}-\frac{\delta}{q_0} & x_{31} & x_{32} & \frac{q_{012}}{q_{0123}}  & \ddots &0  \\
\vdots  & \vdots & \vdots & \ddots & \ddots & 0 \\
\vdots  & \vdots & \vdots & \ddots & \ddots & \frac{q_{0,...,s-2}}{q_{0,...,s-1}} \\
x_{s0}-\frac{\delta}{q_0} & x_{s1} & x_{s2} & x_{s3} & \cdots & x_{s(s-1)}
\end{bmatrix} \end{align*}

Enumerating the columns $0,...,s-1$, after multiplying column $i$ by $q_i$ (thus changing the determinant by a factor of $q_0 \cdots q_{s-1}$) for all $i$, observe that,  by the construction of $e_i$, the sum of all rows except the last one is $0$. For $i=0,...,s-2$ do successively the column operation: add column $i$ to column $i+1$. This will not change the determinant, and observe that by the remark about the row sums, the new matrix will be lower triangular. Thus the determinant will be the product of the diagonal elements.

Diagonal entry number $r$ will be equal to $x_{r0}q_0 -\delta +x_{r1}q_1+...+x_{r(r-1)}q_{r-1}$, which by construction equals $-\frac{q_{0,...,r-1}}{q_{0,...,r}}$. So we get
\[ \frac{1}{q_0 \cdots q_{s-1}}\det(B_s)=(-1)^s\frac{q_0 \cdots q_s}{q_{0,...,s}} \] 
implying that $z_s=q_s$.

The result now follows from the fact that $|z|^2=q_0^2+ \cdots q_n^2$.
\end{proof}

By this result, we have that a polytope with a Euclidean volume of $\frac{\sqrt{q_0^2+\cdots q_n^2}}{n!}$ will have normalized lattice volume equal to $1$, in the lattice spanned by $H$. Using this we have:

\begin{proposition}
The  normalized volume of $P$ is $\frac{\delta^n}{q_0 \cdots q_n}$.
\end{proposition}

\begin{proof}
The edges emanating from $v_0$ are spanned by the vectors
\[ w_i = (-\frac{\delta}{q_0},0,...,\frac{\delta}{q_i},0,...,0), \]
for $i=1,...,n$. The Euclidean volume of $P$ will be $\frac{|w_1 \times \cdots \times w_n|}{n!}$. The corresponding matrix is

\begin{align*}
\begin{bmatrix}
 -\frac{\delta}{q_0} & \frac{\delta}{q_1} & 0 & 0 & \cdots & 0\\
 -\frac{\delta}{q_0} & 0 &  \frac{\delta}{q_2}  & 0 & \cdots & 0\\
 -\frac{\delta}{q_0} & 0 & 0 & \frac{\delta}{q_3}  & \cdots & 0\\
\vdots  & \vdots & \vdots & \ddots & \ddots & 0 \\
 -\frac{\delta}{q_0} & 0 & 0 & 0 & \cdots & \frac{\delta}{q_n}
\end{bmatrix} .
\end{align*}

We see that $w_1 \times \cdots \times w_n = (\frac{\delta^n}{q_1 \cdots q_n}, \frac{\delta^n}{q_0q_2 \cdots q_n}, ...,\frac{\delta^n}{q_0 \cdots \hat{q_i} \cdots q_n}, \cdots, \frac{\delta^n}{q_0 \cdots q_{n-1}})$. This implies 
\[ |w_1 \times \cdots w_n|^2 = \frac{\delta^{2n}q_0^2 + \delta^{2n}q_1^2 + ... + \delta^{2n}q_n^2}{q_0^2 \cdots q_n^2} ,  \]
giving
\[  |w_1 \times \cdots w_n| = \frac{\delta^n}{q_0 \cdots q_n} \sqrt{q_0^2+...+q_n^2} .\]
Combinining this with the normalization yields the result.

\end{proof}

Finally we can return to intersection theory on $\p(q_0,...,q_n)$. Recall that we have $D_0^n = \Vol(P) \frac{q_0^n}{\delta^n}$. Inserting the above gives $D_0^n = \frac {q_0^n}{q_0 \cdots q_n}$. Combining this with the previous calculations we get
\[ E_1 \cdots E_n = \frac {\Pi_{j=1}^n \deg E_j}{q_0^n}D_0^n =\frac {\Pi_{j=1}^n \deg E_j}{q_0^n} \frac {q_0^n}{q_0 \cdots q_n} = \frac {\Pi_{j=1}^n \deg E_j} {q_0 \cdots q_n} \]
thus we are done.

\AtEndDocument{\bigskip{\footnotesize%
  \textsc{Department of Mathematics, University of Oslo,  Norway} \par  
  \textit{E-mail address},  \texttt{berntin@math.uio.no} \par
}}

\bibliography{reference}
\bibliographystyle{alpha}
\end{document}